\documentclass[11pt]{amsart} 
\usepackage{amssymb}
\usepackage{amsmath}
\usepackage{enumitem}
\usepackage{mathrsfs}
\usepackage[english]{babel}
\usepackage[all]{xy}
\usepackage{hyperref}
\usepackage{marginnote}
\usepackage{version}

\usepackage{stmaryrd}

\usepackage{fullpage}

\RequirePackage{mathrsfs} 

\newtheorem{theorem}{Theorem}[section]
\newtheorem{lemma}[theorem]{Lemma}
\newtheorem{proposition}[theorem]{Proposition}
\newtheorem{corollary}[theorem]{Corollary}
\newtheorem{conjecture}[theorem]{Conjecture}

\theoremstyle{definition}

\newtheorem{remark}[theorem]{Remark}

\newtheorem{definition}[theorem]{Definition}

\numberwithin{equation}{section} \numberwithin{figure}{section}

\DeclareMathOperator{\Aut}{Aut}

\DeclareMathOperator{\Hol}{Hol}

\DeclareMathOperator{\Com}{Com}
\DeclareMathOperator{\eucl}{eucl}
\DeclareMathOperator{\codim}{codim}

\DeclareMathOperator{\Id}{Id}
\DeclareMathOperator{\Degen}{Degen}

\DeclareMathOperator{\Sur}{Sur}

\newcommand\BB{\mathbb{B}}

\newcommand\NN{\mathbb{N}}

\newcommand\CC{\mathbb{C}}

\def\P{\mathbb{P}}

\def\bydef{:=}

\usepackage{color}

\definecolor{orange}{rgb}{1,0.5,0}

\title{Complex-analytic intermediate hyperbolicity, and finiteness properties}

 \author{Antoine Etesse}
\address{Aix Marseille Univ, CNRS, Centrale Marseille, I2M, Marseille, France}
\email{antoine.etesse@univ-amu.fr}



\begin{document}

\maketitle

\thispagestyle{empty}

\begin{abstract}
Motivated by the finiteness of the set of automorphisms $\Aut(X)$ of a projective manifold $X$, and by Kobayashi-Ochiai's conjecture that a projective manifold $\dim(X)$-analytically hyperbolic (also known as strongly measure hyperbolic) should be of general type, we investigate the finiteness properties of $\Aut(X)$ for a complex manifold satisfying a (pseudo-) intermediate hyperbolicity property. We first show that a complex manifold $X$ which is $(\dim(X)-1)$-analytically hyperbolic has indeed finite automorphisms group. We then obtain a similar statement for a pseudo-$(\dim(X)-1)$-analytically hyperbolic, strongly measure hyperbolic projective manifold $X$, under an additional hypothesis on the size of the degeneracy set. Some of the properties used during the proofs lead us to introduce a notion of intermediate Picard hyperbolicity, which we last discuss.
\end{abstract}

\renewcommand{\abstractname}{Acknowledgements}
\begin{abstract}
 I would like to thank my supervisor Erwan Rousseau as well as Ariyan Javanpeykar for their useful remarks, advices, and corrections. Specifically, the notion of \emph{intermediate} Picard hyperbolicity emerged from various discussions we had together.
\end{abstract}

 \section{Introduction} 
 It is a standard result that a hyperbolic (in any sense of the term) compact manifold $X$ satisfies the following strong finiteness result: for any compact manifold $Y$, the set of surjective holomorphic maps $\Sur(Y,X)$ is finite. (See \cite{Kobayashi}[Chapter 6, Section 6]; the result is actually true for $X, Y$ complex spaces). In particular, the set of automorphisms $\Aut(X)$ is finite. In a different direction, a compact complex hyperbolic manifold $X$ satisfies the following remarkable extension property: for any compact complex manifold $Y$, and any $A \subsetneq Y$ closed complex subspace of $Y$, every holomorphic map
 \[
 f: Y \setminus A \rightarrow X
 \]
 extends holomorphically through $A$. (See \cite{Kobayashi}[Chapter 6, Section 3]; the result is actually valid for $X$ a complex space).
 
 In \cite{eisenman1970intrinsic}, Eisenmann introduced the so-called \emph{Kobayashi-Einsenmann pseudo-metrics}, leading to weaker forms of hyperbolicity, namely \emph{$\ell$- analytic hyperbolicity}, $1 \leq \ell \leq \dim(X)$, that we sometimes refer to as ``intermediate hyperbolicity``. The case $\ell=1$ leads to the usual notion of hyperbolicity, and  a special focus was first given to the weakest form (i.e. for $\ell=\dim(X)$), often called \emph{(strong)-measure hyperbolicity} (see Section \ref{section: basics} for definitions). Restricting ourselves to the realm of projective manifolds, it was proved by Kobayashi and Ochiai that a projective manifold $X$ of general type is (strongly-)measure hyperbolic (see e.g. \cite{Kobayashi}[Chapter 7, Section 4]), and they further conjectured that the converse should be true. 
 
 The finiteness property  that we described above, valid for hyperbolic manifolds, turns out to be satisfied by projective manifolds of general type as well (See \cite{Kobayashi}[Chapter 7, Section 6]); as for the property of extension, the following weaker form is valid for projective manifolds of general type: 
 \begin{center}
(*) \ \ \ \ 
if $\dim(Y) \geq \dim(X)$, then any \emph{non-degenerate} holomorphic map $f: Y \setminus A \rightarrow X$ extends \emph{meromorphically} through $A$,
 \end{center}
 where $f$ is said to be non-degenerate if its differential is of a maximal rank at some point (see Section \ref{subs: pseudo} where we recall the definition of a meromorphic map).
 
 In light of the conjecture we mentioned above, it is natural to expect that such finiteness and extension properties should be satisfied by measure hyperbolic projective manifold. As such kind of properties are actually satisfied by hyperbolic manifold, not necessarily projective, we may even be tempted to drop the projective hypothesis. In any case, this leads naturally to the following two conjectures. (See also \cite{graham1985}).
 \begin{conjecture}[Finiteness property]
 \label{conjecture: finiteness}
 Let $X$ be a projective manifold $\dim(X)$-analytically hyperbolic. Then $\Aut(X)$ is finite.
 \end{conjecture}

\begin{conjecture}[Extension property]
\label{conjecture: extension}
Let $X$ be a projective $\dim(X)$-analytically hyperbolic manifold,  let $Y$ be a projective manifold of dimension $\dim(Y)\geq \dim(X)$, and let $H \subset Y$ be an hypersurface. Then any non-degenerate holomorphic map
\[
f: Y \setminus H \rightarrow X
\]
extends to a meromorphic map $f: Y \dashrightarrow X$.
\end{conjecture}
It is worth noting here that the usual notion of (pseudo-)hyperbolicity, defined for complex spaces, has its analogues in the algebraic and arithmetic worlds, and that such finiteness properties were investigated in these settings: see \cite{BKV} and \cite{JKa} for \emph{algebraic hyperbolicity}, and \cite{JAut, JXie} for \emph{pseudo algebraic hyperbolicity} and \emph{arithmetic hyperbolicity}.
\newline

In this paper, we investigate the finiteness properties for intermediate (analytic) hyperbolicities, and prove first the following theorem.
 \begin{theorem}
 \label{thm: mainthm1}
 Let $X$ be a compact complex manifold which is $(\dim(X)-1)$-analytically hyperbolic. Then $\Aut(X)$ is finite.
 \end{theorem}
 Note that $X$ is not supposed projective, and that unfortunately we were not able to obtain the suspected optimal hypothesis, i.e. $\dim(X)$-analytic hyperbolicity.
 
Following Lang's terminology, each intermediate hyperbolicity has its \emph{pseudo}-version, where a closed degeneracy set is allowed for the Kobayashi-Eisenmann (pseudo-)metrics (see Section \ref{section: basics} for definitions). Depending on whether we want the degeneracy set to be analytically closed, or simply topologically closed, we will write \emph{pseudo$^{an}$} or \emph{pseudo$^{top}$}.
In view of Lang's conjecture that a projective manifold of general type should be pseudo$^{an}$-hyperbolic, it is expected that, in the realm of projective manifolds, the intermediate pseudo$^{an}$-hyperbolicities should be equivalent. For more details on Lang's conjectures, and their analogues in the algebraic and arithmetic settings, see e.g. \cite{JBook}.

We were then naturally lead to investigate the question  of finiteness of $\Aut(X)$ under the weaker assumptions of $\dim(X)$-analytic hyperbolicity and pseudo$^{top}$ / pseudo$^{an}$-$(\dim(X)-1)$-analytic hyperbolicity, and were able to prove the following
\begin{theorem}
\label{thm: mainthm2}
Let $X$ be \emph{projective} manifold, $\dim(X)$-analytically hyperbolic and $(\dim(X)-1)$-analytically hyperbolic modulo a \emph{topologically} closed subset $\Delta=\Delta_{X,\ell}$ of zero $(2\dim(X)-3)$ Haussdorf dimension. Then $\Aut(X)$ is finite.
\end{theorem}
Note that this time, the manifold is supposed to be projective, and that any Zariski closed proper subset $\Delta$ of codimension $\geq 2$ satisfies the assumptions of the theorem. The hypothesis on the degeneracy set may seem odd: this is a technical hypothesis necessary to apply a very general extension theorem due to Siu \cite{siu1975extension}. 

In the case of pseudo$^{an}$-$(\dim(X)-1)$-analytic hyperbolicity, i.e. when $\Delta$ is algebraic, the extension property needed in the course of the proof is in fact exactly the extension property (*) we saw above. Accordingly, and interestingly enough, the questions on finiteness are connected to the ones on extension. Whereas this seemed to us to be a much harder task to relate intermediate hyperbolicities to this extension property, in view of the recent papers \cite{JKuch} and \cite{deng2020big}, we realized that there is a natural notion of \emph{intermediate Picard hyperbolicity}  that precisely encapsulates these extension properties. In this framework, the scheme of proof of \ref{thm: mainthm2} allowed us to obtain the following theorem (see Section \ref{section: Picard} for the intermediate notions of $\ell$-Picard hyperbolicity)
\begin{theorem}
\label{thm: mainthm3}
Let $X$ be a projective manifold, $\dim(X)$-Picard hyperbolic, $\dim(X)$-analytically hyperbolic and pseudo$^{an}$-$(\dim(X)-1)$-analytically hyperbolic. Then $\Aut(X)$ is finite.
\end{theorem}

The paper is organized as follows. In the first part, we recall the notions of pseudo$^{top/an}$-intermediate hyperbolicities, and prove some basic results regarding them. In the second part, we start by proving the finiteness result of $\Aut(X)$ in the absolute case (Theorem \ref{thm: mainthm1}), and then move onto the proof in the pseudo case (Theorem \ref{thm: mainthm2}). In the third and last part, we introduce the notions of \emph{intermediate Picard hyperbolicities}, give some of their basic properties, and prove Theorem \ref{thm: mainthm3}.

\section{Automorphisms of $\ell$-analytically hyperbolic manifolds}
 \label{analytic hyperbolicity}

This section, motivated by Conjecture \ref{conjecture: finiteness}, is devoted to the proof of Theorem \ref{thm: mainthm1} and its pseudo-version Theorem \ref{thm: mainthm2}. We start by proving some basic properties on intermediate hyperbolicity, in particular that $\ell$-analytic hyperbolicity modulo $\Delta$ implies $(\ell+1)$-analytic hyperbolicity modulo $\Delta$ (See also \cite{graham1985}, where it is proved in the absolute case), and then move to the two proofs in question. We end this section by some conditions ensuring intermediate hyperbolicity, and hence the finiteness of $\Aut(X)$ in suitable cases.

\subsection{Basics on $\ell$-analytic hyperbolicity}
 \label{section: basics}
Throughout this section, we let $X$ be a  {compact} complex manifold of dimension $p$ and $1 \leq \ell \leq p$ be an integer. Following \cite{Demailly} and \cite{graham1985}, define the  \emph{Kobayashi-Eisenmann pseudo-metric $e_{\ell}$} on \emph{pure} multi-vectors of $\bigwedge^{\ell}TX$ by setting for $\xi=v_{1}\wedge \dotsc \wedge v_{\ell} \in \bigwedge^{\ell}T_{x}X, x \in X$:
\[
e_{\ell}(x, \xi)
=
\{
\inf \frac{1}{R} \ | \ \exists \ f: \mathbb{B}^{\ell} \rightarrow X \ \text{holomorphic}, \  f(0)=x, \ f_{*}\big(\frac{\partial}{\partial z_{1}} \wedge \dotsc \wedge \frac{\partial}{\partial z_{\ell}}\big)=R \xi \}.
\]
For $\Delta \subsetneq X$ a closed subset of $X$ (of empty interior), the compact complex manifold $X$ is  said  $\ell$-analytically hyperbolic modulo $\Delta$ if and only if there exists a continuous Finsler pseudo-metric $\omega$ defined on pure multi-vectors of $\bigwedge^{\ell}TX$, which is a metric on $X \setminus \Delta$ and such that
$e_{\ell} \geq \omega$
on $X$. We will denote $\Delta_{X, \ell}$ the smallest closed subset satisfying such a property, commonly referred to as the \emph{exceptional locus} (with respect to the pseudo-metric $e_{\ell}$).
\begin{remark}
\label{remark: smallest}
 Such a set is well defined since if $X$ is $\ell$-hyperbolic modulo $\Delta_{1}$ and $\ell$-hyperbolic modulo $\Delta_{2}$, then $X$ is $\ell$-hyperbolic modulo $\Delta_{1} \cap \Delta_{2}$: for $i=1,2$, take $\omega_{i}$ a continuous Finsler pseudo-metric on (pure multi-vectors of) $\bigwedge^{\ell}TX$, which is a metric on  $X\setminus \Delta_{i}$, and such that $e_{\ell} \geq \omega_{i}$ on $X$, and consider
\[
\omega= \frac{\omega_{1} + \omega_{2}}{2},
\]
which is a continuous Finsler pseudo-metric on (pure multi-vectors of) $\bigwedge^{\ell}TX$, a metric on $X \setminus (\Delta_{1} \cap \Delta_{2})$, and satisfies $e_{\ell} \geq \omega$ on $X$.
\end{remark}
We will say that $X$ \emph{pseudo$^{top}$} (resp. \emph{pseudo$^{an}$}) $\ell$-analytically hyperbolic if $\Delta_{X,\ell}$ is of empty interior (resp. if $\Delta_{X, \ell}$ is a proper analytically closed subset). Note that if $\ell=\dim X$ and $\Delta:=\emptyset$, such a manifold is called \emph{strongly measure hyperbolic} in \cite{Kobayashi} and more generally \emph{strongly $\ell$-measure hyperbolic} in \cite{graham1985} for $1\leq \ell\leq\dim X$ and $\Delta:=\emptyset$.

A first basic observation is that, for $f: \CC \times \mathbb{B}^{\ell-1} \rightarrow X$ a holomorphic map, the pseudo-metric $e_{\ell}$ degenerates along the image of every non-degenerate point of $f$. Note that we say that $f$ is   \emph{non-degenerate} if there exists at least one point (and thus a non-empty Zariski open subset as being degenerate for an holomorphic map is an analytically closed condition) at which $f$ is non-degenerate.
If $\ell=1$, by Brody's reparametrization lemma (see e.g. \cite{Kobayashi}[Lemma 3.6.2]), the existence of a degeneracy locus for $e_{1}$ implies the existence of a non-constant holomorphic map $f: \CC \rightarrow X$, commonly referred to as an \emph{entire curve} in $X$. Let us make two comments on that. First, the degeneracy locus of $e_{1}$ need not be equal to the union of the closure of the image of all entire curves. The reason for this lies in the way one constructs an entire curve via Brody's lemma: we have no control on the image of the entire curve, and in particular, we can not ensure that there exists an entire curve passing arbitrarly close to a point where $e_{1}$ degenerates. Second, we see that it implies that a compact complex manifold $X$ is $1$-analytically hyperbolic if and only if $X$ contains no entire curves ($X$ is commonly called \emph{Brody-hyperbolic}). However, in the case where we allow a degeneracy set for $e_{1}$,  there is to our knowledge no clean relationship between the geometry of entire curve and $1$-analytic hyperbolicity modulo $\Delta$. Even worse, in the case $p>1$, the non-existence of non-degenerate holomorphic maps $f: \CC\times \mathbb{B}^{\ell-1} \rightarrow X$ does not imply, a priori, that $X$ is $\ell$-analytically hyperbolic.
Also, note that, in the case where $p=1$, the notion of $1$-analytic hyperbolicity modulo $\Delta$ is the same as the notion of hyperbolicity modulo $\Delta$ (in the sense of Kobayashi), where we recall that $X$ is hyperbolic modulo $\Delta$ if and only if the Kobayashi metric $d_{X}$ satisfies 
\[
d_{X}(p,q)=0 \Longleftrightarrow p=q, \ \text{or} \ p, q \in \Delta.
\]
Indeed, one implication follows from a result of Royden (\cite{royden1971remarks}), stating the Kobayashi pseudo-distance $d_{X}$ is obtained as the integration of the infinitesimal pseudo-metric $e_{1}$, whereas the reverse implication follows for instance from \cite{Kobayashi}[Cor. (3.5.41)].

\subsection{From $\ell$ to $\ell+1$}
Equip $X$ with any smooth hermitian metric with its induced Riemaniann metric on $X$. Moreover, for $1\leq \ell \leq p$,   equip the vector bundles $\bigwedge^{\ell}TX$ with smooth (Finsler) metrics defined on pure multi-vectors that we denote $F_{\ell}$. Here,  $F_{1}$ is the hermitian metric.

Let $U \overset{\varphi}{\simeq} \mathbb{B}^{p}$ be any trivializing open set, and denote $(\frac{\partial}{\partial z_{i}})_{1 \leq i \leq p}$ the standard vector fields on $\mathbb{B}^{p}$. Let $1\leq \ell \leq p$ be an integer.
For $v_{1} \wedge \dotsc \wedge v_{\ell} \in \bigwedge^{\ell}T_{x}X$ and $x \in U$, denote 
\[
A(x; v_{1}, \dotsc, v_{\ell})
\]
the $(\ell\times p)$-matrix which represents $\big((\varphi)_{*}(v_{1}), \dotsc, (\varphi)_{*}(v_{\ell})\big)$ in the basis $(\frac{\partial}{\partial z_{1}}, \dotsc, \frac{\partial}{\partial z_{p}})$. We   define a smooth Finsler metric on pure multi-vectors of $\bigwedge^{\ell}TX_{| U}$ by setting
\[
h_{\ell,U}(x; v_{1} \wedge \dotsc \wedge v_{\ell})
\bydef
\max \{ |\det(M)| \ | \ \text{$M$ is a $\ell\times \ell$ extracted matrix of $A(x; v_{1}, \dotsc, v_{\ell})$} \}.
\]
Note that this metric is well-defined. Indeed, if $v_{1}'\wedge \dotsc \wedge v_{\ell}'=v_{1}\wedge \dotsc \wedge v_{\ell}$, then there exists a matrix $Q$ in $\mathrm{GL}_{\ell}(\CC)$ with $\det(Q)=1$  such that \[QA(x; v'_{1}, \dotsc, v'_{\ell})=A(x; v_{1}, \dotsc, v_{\ell}).\]
 Since both $F_{\ell}$ and $h_{\ell,U}$ are continuous on $U$, for any compact subset $K \subset U$,  there exist  real numbers $0<c\leq C $ such that, for all $x \in K$ and all $\xi \in \bigwedge^{\ell}T_{x}X$ pure multi-vector, the inequalities
\[
c\cdot \big|\xi\big|_{h_{\ell,U}(x)} 
\leq
\big|\xi\big|_{F_{\ell}(x)} 
\leq 
C \cdot \big|\xi\big|_{h_{\ell,U}(x)}
\] hold.
 Accordingly, when dealing with questions of convergence towards $0$ or divergence to $\infty$ with respect to the metric $F_{\ell}$ of sequences of pure multi-vectors, we can always use the local metric $h_{\ell,U}$.

To prove Proposition \ref{prop:l_to_lplusone} below, we will use the following  basic lemma, which is merely a reformulation of what $\ell$-analytic hyperbolicity modulo $\Delta$  means. 
 
\begin{lemma}
\label{lemma: basic lemma1}
Let $\Delta \subsetneq X$ be a closed set and let $1 \leq \ell \leq p$. If $X$ is not $\ell$-analytically hyperbolic modulo $\Delta$, then there exists $h_{n}: \mathbb{B}^{\ell} \rightarrow X$ holomorphic maps such that
\[
\big|(h_{n})_{*}(\frac{\partial}{\partial z_{1}} \wedge \dotsc \wedge \frac{\partial}{\partial z_{\ell}})\big|_{F_{\ell}(x_{n})} \rightarrow \infty,
\]
with $x_{n}=h_{n}(0) \rightarrow x \notin \Delta$.

Reciprocally, if there exists $h_{n}: \mathbb{B}^{\ell} \rightarrow X$ holomorphic maps, and $z_{n} \rightarrow 0$ such that
\[
\big|(h_{n})_{*}(\frac{\partial}{\partial z_{1}} \wedge \dotsc \wedge \frac{\partial}{\partial z_{\ell}})\big|_{F_{\ell}(x_{n})} \rightarrow \infty,
\]
with $x_{n}=h_{n}(z_{n}) \rightarrow x \notin \Delta$, then $X$ is not $\ell$-analytically hyperbolic modulo $\Delta$.
\end{lemma}
\begin{proof}
Suppose first that $X$ is not $\ell$-analytically hyperbolic modulo $\Delta$. Consider $(U_{n})_{n \in \NN}$ an exhaustion of relatively compact domains of $X \setminus \Delta$, i.e. $\overline{U_{n}} \subset U_{n+1} \subset X \setminus \Delta$ for all $n \in \NN$, and $\bigcup_{n \in \NN} U_{n}= X \setminus \Delta$. Observe then that the following property is not satisfied
\begin{equation}
\label{eq: l => l+1}
\begin{aligned}
\forall n \in \NN, \ \exists\delta_{n} > 0 \ \text{such that} \  e_{\ell} \geq \delta_{n} F_{\ell}\  \text{on} \  U_{n}.
\end{aligned}
\end{equation}
Indeed, otherwise by constructing a nonnegative continuous function $\varphi$ on $X$, which is equal to $0$ on $\Delta$, and satisfies $0 < \varphi \leq \delta_{n}$ on $U_{n}$, we deduce that
\[
\varphi F_{\ell} \leq e_{l}
\]
on $X$. As $\varphi F_{\ell}$ is a continuous Finsler pseudo-metric on (pure multi-vectors of) $\bigwedge^{\ell}TX$, which is a metric on $X\setminus \Delta$, this contradicts the fact is not $\ell$-analytically hyperbolic modulo $\Delta$.
Now, as \eqref{eq: l => l+1} is not satisfied, we deduce in particular the existence of a sequence of points $x_{n} \rightarrow x \in X \setminus \Delta$, aswell as a sequence of pure multi-vectors $\xi_{n} \in \bigwedge^{\ell} T_{x_{n}}X$, $|\xi_{n}|_{F_{\ell}(x_{n})}=1$, such that $e_{\ell}(x_{n}; \xi_{n}) \rightarrow 0$. By definition of $e_{\ell}$, this in turn gives a sequence of holomorphic maps $h_{n}: \mathbb{B}^{\ell} \rightarrow X$, $h_{n}(0)=x_{n}$, satisfying
\[
(h_{n})_{*}(\frac{\partial}{\partial z_{1}}\wedge \dotsc \wedge \frac{\partial}{\partial z_{\ell}})
=
R_{n} \xi_{n},
\]
with $R_{n} \rightarrow \infty$, so that the first implication is proved.

In the other direction, observe first that we can always suppose that $z_{n}=0$ for all $n \in \NN$: indeed, for all $n$, pick $\tau_{n} \in \Aut(\mathbb{B}^{\ell})$ such that $\tau_{n}(0)=z_{n}$, and consider instead
$\tilde{h_{n}}=h_{n}\circ \tau_{n}$; the announced fact then follows from the equality
\[
\big|(\tilde{h_{n}})_{*}(\frac{\partial}{\partial z_{1}} \wedge \dotsc \wedge \frac{\partial}{\partial z_{\ell}})\big|_{F_{\ell}(x_{n})}
=
\big|\det(d\tau_{n})_{0}\big| \cdot \big|({h_{n}})_{*}(\frac{\partial}{\partial z_{1}} \wedge \dotsc \wedge \frac{\partial}{\partial z_{\ell}})\big|_{F_{\ell}(x_{n})},
\]
and the fact that $\big|\det(d\tau_{n})_{0}\big|$ remains bounded (as $z_{n}$ stays away of $\partial \mathbb{B}^{\ell}$).

If $X$ were $\ell$-analytically hyperbolic modulo $\Delta$, then by definition there would exist a continuous Finsler metric $\omega$ on (pure multi-vectors of) $\bigwedge^{l} TX$, which is a metric on $X \setminus \Delta$, such that
\[
e_{\ell} \geq \omega
\]
on $X$. In particular, letting $U$ be a relatively compact neighborhood of $x$ in $X \setminus \Delta$, there exists $\delta>0$ such that for any $x \in U$ and $\xi=v_{1} \wedge \dotsc \wedge v_{\ell} \in \bigwedge^{\ell} T_{x}X$, we have the inequality
 \begin{equation}\begin{aligned}
\label{inequality1}
|\xi|_{e_{\ell}(x)} 
\geq 
|\xi|_{\omega(x)} 
\geq
\delta |\xi|_{F_{\ell}(x)}.
\end{aligned}\end{equation}
Now, the sequence
\[
\xi_{n}
=
\frac{(h_{n})_{*}(\frac{\partial}{\partial z_{1}} \wedge \dotsc \wedge \frac{\partial}{\partial z_{\ell}})}
{\big|(h_{n})_{*}(\frac{\partial}{\partial z_{1}} \wedge \dotsc \wedge \frac{\partial}{\partial z_{\ell}})\big|_{F_{\ell}(x_{n})}}
\in \bigwedge^{\ell}T_{x_{n}}X
\]
 satisfies $|\xi_{n}|_{F_{\ell}(z_{n})}=1$  by construction,
and by the very definition of the metric $e_{\ell}$, we have:
\[
\big|\xi_{n}\big|_{e_{\ell}(x_{n})} 
\leq 
\frac{1}{\big|(h_{n})_{*}(\frac{\partial}{\partial z_{1}} \wedge \dotsc \wedge \frac{\partial}{\partial z_{\ell}})\big|_{F_{\ell}(x_{n})}}
\longrightarrow 0,
\]
which contradicts the inequality \eqref{inequality1}, as for all $n$ large enough, $x_{n}$ lies in $U$. Therefore $X$ is not $\ell$-analytically hyperbolic modulo $\Delta$.
\end{proof}

 We will also use the following   elementary lemma in linear algebra.
\begin{lemma}
\label{lemma: basic linear algebra}
Let $A_1, A_2, \ldots $ be a sequence of matrices in $  \mathrm{GL}_{\ell}(\CC)$ with $|\det(A_{n})| \underset{n \rightarrow \infty}{\longrightarrow} \infty$. Then, the sequence $(A_n)$ has a subsequence $(A_{\varphi(n)})$ such that the norm of some fixed minor of $(A_{\varphi(n)})$ tends to infinity.  
\end{lemma}
\begin{proof}
This follows immediately from the equality 
\[
|\det(A_{n})|^{\ell-1}=|\det(\Com(A_{n}))| \rightarrow \infty,
\]
where $\Com(A_{n})$ is the co-matrix of $A_{n}$, i.e., the matrix whose entries are  the minors of $A_{n}$.
\end{proof}
We now come to the following proposition (see also \cite{graham1985}, Proposition 2.17)
\begin{proposition}
\label{prop:l_to_lplusone}
If $X$ is $\ell$-analytically hyperbolic modulo $\Delta$, then $X$ is $(\ell+1)$-analytically hyperbolic modulo at least $\Delta$.
\end{proposition}
\begin{proof}
Assume that $X$ is not $(\ell+1)$-analytically hyperbolic modulo $\Delta$. We show that it is not $\ell$-analytically hyperbolic modulo $\Delta$. By our assumption and the Lemma \ref{lemma: basic lemma1}, we can find a sequence of points $x_{n} \in X$, with $x_{n} \rightarrow x \in X\setminus \Delta$,  a sequence of holomorphic maps $f_{n}: \mathbb{B}^{\ell+1} \rightarrow X$ with  $f_{n}(0)=x_{n}$, and a sequence of pure multi-vectors $\xi_{n} \in \bigwedge^{\ell+1}T_{x_{n}}X$ with $|\xi_{n}|_{F_{\ell+1}(x_{n})}=1$ such that
\[
(f_{n})_{*}(\frac{\partial}{\partial z_{1}} \wedge \dotsc \wedge \frac{\partial}{\partial z_{\ell+1}})=R_{n} \xi_{n},
\]
where $R_n$ is a sequence of real numbers such that $R_{n} \rightarrow \infty$. 
 
  Let $U \simeq \mathbb{B}^{p}$ be a trivializing open set around $x$, relatively compact in $X \setminus \Delta$. We now  use the local Finsler metrics $h_{i, U}$ for $1\leq i \leq p$. For $1\leq j \leq \ell+1$, define $v_{j,n}:=(df_{n})_{0}(\frac{\partial}{\partial z_{j}})$; using the notations introduced earlier, we deduce (up to passing to a subsequence) that the determinant of a fixed extracted square matrix of maximal size, i.e. size $(\ell+1)$, of $A(x_{n}; v_{1,n}, \dotsc, v_{\ell+1,n}) \in M_{p, \ell+1}(\CC)$,  tends to infinity in norm. By Lemma \ref{lemma: basic linear algebra}, we deduce (up to passing to a subsequence) that the determinant of a fixed  square matrix of size $\ell$ of $A(x_{n}; v_{1,n}, \dotsc, v_{\ell+1,n})$ also tends to infinity in norm; we can always suppose (up to re-ordering the variables) that this matrix is obtained by keeping the first $\ell$ columns fixed. 
Now, consider the sequence of maps
\[
h_{n}: \mathbb{B}^{\ell} \rightarrow X, \ (z_{1}, \dotsc, z_{\ell}) \mapsto f_{n}(z_{1}, \dotsc, z_{\ell},0).
\]
Note that, by  construction, this sequence satisfies
 \[\big|(h_{n})_{*}(\frac{\partial}{\partial z_{1}} \wedge \dotsc \wedge \frac{\partial}{\partial z_{\ell}})\big|_{h_{\ell,U}(x_{n})} \rightarrow \infty.\] 
In particular, it follows that 
\[\big|(h_{n})_{*}(\frac{\partial}{\partial z_{1}} \wedge \dotsc \wedge \frac{\partial}{\partial z_{\ell}})\big|_{F_{\ell}(x_{n})} \rightarrow \infty.\] 
It now follows from    Lemma \ref{lemma: basic lemma1} that $X$ is not $\ell$-analytically hyperbolic modulo $\Delta$. This concludes the proof.
\end{proof}

\subsection{Finiteness of automorphism groups}
\label{sse: finiteness}
In this section we prove that, if  $X$ is  a compact complex manifold which is $(\dim(X)-1)$-analytically hyperbolic, then the automorphism group $\Aut(X)$ of $X$ is finite. As before, we start with a basic lemma in linear algebra.

\begin{lemma}
\label{lemma: basic lemma2} Let $|.|$ be a norm on $\CC^{p}$, 
let $\varepsilon > 0$ be a real number, and let $A \in M_{p}(\CC)$  be such that
\[
\underset{|v|=1}{\inf} |Av| \geq \varepsilon. 
\]
Then, there exists a real number $K>0$ independant of $A$ such that
\[
K \varepsilon^{p-1}\ ||A|| 
\leq 
|\det(A)|,
\]
where the   norm $||A||$ is computed with respect to $|.|$.  
\end{lemma}
\begin{proof}
Let us first consider the case of the usual hermitian norm on $\CC^{p}$. Note that, by our assumption, the matrix $A$ is invertible. Now, 
consider the polar decomposition  
\[
A=US
\]
of $A$, where $U$ is a unitary matrix and $S$ is a positive definite hermitian matrix.  As $|\det(A)|=|\det(S)|$, and $|Av|=|USv|=|Sv|$ for any $v \in \CC^{N}$, it suffices to prove the lemma for $S$, i.e., we may and do assume that $A$ is a positive-definite hermitian matrix. 

Let $\varepsilon \leq \lambda_{1} \leq \dotsc \leq \lambda_{p}$ be the (real) eigen values of $A$, where the first inequality comes from the hypothesis $\underset{|v|=1}{\inf} |Av| \geq \varepsilon$. Let $(v_{1}, \dotsc, v_{p})$ be an orthonormal basis of $\CC^{p}$ such that $A$ is diagonal with respect to this basis,   and let $v \in \CC^{p}$ with $|v|=1.$ Then, we may  write   
\[
v
=
\sum\limits_{k=1}^{p} \alpha_{k}v_{k},
\]
where $\alpha_{k} \in \CC$ and $\sum\limits_{k=1}^{p} |\alpha_{k}|^{2}=1$. Therefore, we have that
\[
|Sv|
=
\sqrt{\sum\limits_{k=1}^{p} |\alpha_{k}|^{2}\lambda_{k}^{2}}
\leq 
\lambda_{p}
\leq
\lambda_{p} \frac{\lambda_{p-1}}{\varepsilon}\dotsc \frac{\lambda_{1}}{\varepsilon}
=
\frac{1}{\varepsilon^{p-1}}\det(S).
\]
This shows that the statement of the lemma holds when $|.|$ is the usual hermitian norm on $\mathbb{C}^p$ with $K=1$. 

To conclude the proof, let  $|.|$ be a norm on $\mathbb{C}^p$. Then,  there exist real numbers $0<c \leq C $ such that
\[
c. |.|_{\eucl} \leq |.| \leq C |.|_{\eucl},
\]
where $|.|_{\eucl}$ is the usual hermitian norm. Then, by the above, it follows readily that the lemma holds with  $K:=(\frac{c}{C})^{p}$.
\end{proof}

To prove the desired finiteness of $\Aut(X)$ when $X$ is $(\dim X-1)$-analytically hyperbolic, we will show that $\Aut(X)$ is discrete and compact. The discreteness is well-known, and we record it in the following lemma.
\begin{lemma}
\label{lem:discreteness_measure}
Let $X$ be a compact complex manifold. If $X$ is pseudo$^{top}$-$\dim X$-analytically hyperbolic, then $\Aut(X)$ is discrete.
\end{lemma}
\begin{proof}
As $X$ is a compact complex variety (of dimension $p$), it is known that $\Aut(X)$ is a complex Lie-group by a theorem of Bochner-Montgomery (\cite{bochner1947groups}).  Accordingly, it must be discrete since otherwise, by considering a non-zero complex vector field on $X$, one can construct a non-degenerate holomorphic map
\[
F: \CC \times \mathbb{B}^{p-1} \rightarrow X.
\]
But this contradicts the pseudo $\dim(X)$-analytic hyperbolicity of $X$ as the image of such a map would have a non-empty interior and would be included in the degeneracy set of $e_{p}$.
\end{proof}

We now prove the finiteness of $\Aut(X)$, assuming $X$ is $(\dim X-1)$-analytically hyperbolic, by showing that it is compact.
\begin{proof}[Proof of Theorem \ref{thm: mainthm1}]
To prove the theorem, by Lemma \ref{lem:discreteness_measure}, it is   enough to prove that $\Aut(X)$ is compact.
To show this compactness, 
let $(f_{n})_{n \in \NN}$ be a sequence in $\Aut(X)$. By classic arguments involving Arzela-Ascoli theorem and convergence of holomorphic maps, in order to  prove that $(f_{n})_{n \in \NN}$ has an adherence value in $\Aut(X)$ (i.e. a converging subsequence) it  suffices to prove that 
$$ \underset{n \in \NN}{\liminf} \ \underset{x \in X}{\max} \ ||(df_{n})_{x}|| < \infty.$$
Arguing by contradiction, up to passing to a suitable subsequence, we can suppose that there exists $x_{n} \rightarrow x$, with $y_{n}=f_{n}(x_{n}) \rightarrow y$, satisfying $||(df_{n})_{x_{n}}|| \rightarrow \infty$, where the norm is computed with respect to $F_{1}$.

Let $U \overset{\varphi}{\simeq} \mathbb{B}^{p}$ be  a trivializing open set around $x$ with $\varphi(x)=0$, and let  $V \overset{\psi}{\simeq} \mathbb{B}^{p}$ be a trivializing open set around $y$ with $\psi(y)=0$. Replacing $x_n$ and $y_n$ by subsequences if necessary, we  may  assume that, for every $n$, $x_{n} \in U$ and $y_{n} \in  V$.  Define $z_{n}=\varphi(x_{n})$  and define $g_{n} \bydef \psi \circ f_{n} \circ \varphi^{-1}$; note that this function  $g_n$ is  well-defined in a neighborhood of $z_{n}$ (but maybe not on the whole set $\mathbb{B}^{p}$).

From now on, we are going to use the metrics $h_{\ell,U}$ and $h_{\ell,V}$ defined via the trivializations $\varphi$ and $\psi$, as described in the beginning of this section. 
Let us show that there exist real numbers $0<m \leq M $ such that, for every $n \in \NN$,
\[
m
\leq
\big|(f_{n}\circ \varphi^{-1})_{*}\big(\frac{\partial}{\partial z_{1}} \wedge \dotsc \wedge \frac{\partial}{\partial z_{p}}\big)\big|_{h_{p,V}(y_{n})}
\leq
M.
\]
Indeed, the upper bound follows immediately from the Lemma \ref{lemma: basic lemma1}. To prove the lower bound, let $A_{n} \in \mathrm{GL}_{p}(\CC)$ be the matrix representing the linear map $(dg_{n})_{z_{n}}$ with respect to the canonical basis of $\CC^{p}$, and observe that, by the definition of $h_{p,V}$,   the equality
\[
\big|(f_{n}\circ \varphi^{-1})_{*}\big(\frac{\partial}{\partial z_{1}} \wedge \dotsc \wedge \frac{\partial}{\partial z_{p}}\big)\big|_{h_{p,V}(y_{n})}
=
|\det(A_{n})|
\] holds.
Consider   $f_{n}^{-1}$ and $g_{n}^{-1}=\varphi \circ f_{n}^{-1} \circ \psi^{-1}$. Now, the matrix representing the linear map $(dg_{n}^{-1})_{g_{n}(z_{n})}$ with respect to the canonical basis of $\CC^{p}$ is $A_{n}^{-1}$, and similarly we have that
\[
\big|(f_{n}^{-1}\circ \psi^{-1})_{*}\big(\frac{\partial}{\partial z_{1}} \wedge \dotsc \wedge \frac{\partial}{\partial z_{p}}\big)\big|_{h_{p,U}(x_{n})}
=
|\det(A_{n}^{-1})|
=
\frac{1}{|\det(A_{n})|}.
\]
Thus, by applying once again   Lemma \ref{lemma: basic lemma1}, we obtain the desired lower bound.

Now, we show that, up to extracting and re-ordering the variables,  there exists $v_{n}=(\alpha_{n}^{1}, \dotsc, \alpha_{n}^{p}) \rightarrow v=(\alpha^{1}, \dotsc, \alpha^{p}) \in \CC^{p}$ with $|v_{n}|=|\alpha_{n}^{p}|=1$ such that
\[
\big|A_{n}v_{n}\big| \rightarrow 0,
\]
where we consider the sup norm on $\CC^{p}$, i.e., $|(v_{1}, \dotsc, v_{\ell})|= \underset{1 \leq i \leq \ell}{\max} |v_{i}|$.
Indeed, recall   that $||(df_{n})_{x_{n}}|| \rightarrow \infty$, so that $||A_{n}|| \rightarrow \infty$. Arguing by contradiction, there must exist $\varepsilon > 0$ such that, for every $v \in \CC^{p}$ with $|v|=1$ and every $n \in \NN$, 
\[
|A_{n}v| \geq \varepsilon.
\]
In this case, by   Lemma \ref{lemma: basic lemma2}, there is a real number $K>0$ such that $|\det(A_{n})| \geq K\varepsilon^{p-1}||A_{n}||$. However, as the latter diverges to infinity, this implies that $|\det(A_n)| \rightarrow \infty$. This  contradicts the $(\dim X)$-analytic hyperbolicity by   Lemma \ref{lemma: basic lemma1}, as 
\[\big|(f_{n}\circ \varphi^{-1})_{*}\big(\frac{\partial}{\partial z_{1}} \wedge \dotsc \wedge \frac{\partial}{\partial z_{p}}\big)\big|_{h_{p,V}(y_{n})}
=
|\det(A_{n})|.\]
To conclude the proof, note that
\begin{equation*}
\begin{aligned}
\big|(f_{n}\circ \varphi^{-1})_{*}(\frac{\partial}{\partial z_{1}} \wedge \dotsc \wedge \frac{\partial}{\partial z_{p}})\big|_{h_{p,V}(y_{n})}
&
=
\frac{1}{|\alpha_{n}^{p}|} \big|(f_{n}\circ \varphi^{-1})_{*}(\frac{\partial}{\partial z_{1}} \wedge \dotsc \wedge \frac{\partial}{\partial z_{p-1}}\wedge v_{n})\big|_{h_{p,V}(y_{n})}
\\
&
=
|\det(B_{n})|,
\end{aligned}
\end{equation*}
where  $B_{n}$  is defined to be the matrix $A_{n}$ except that its  last column (equal to $A_{n}\frac{\partial}{\partial z_{p}}$) has been replaced by $A_{n}v_{n}$. 
Now, developing the determinant with respect to the last column, since $A_{n}v_{n} \rightarrow 0$ and   $|\det(A_{n})|$ is bounded from below by $m>0$, we deduce that at least one minor of the matrix representing $\big(A_{n}\frac{\partial}{\partial z_{1}}, \dotsc, A_{n}\frac{\partial}{\partial z_{p-1}}\big)$ has its determinant tending to infinity in norm. This shows that 
\[
\big| (f_{n}\circ \varphi^{-1})_{*}(\frac{\partial}{\partial z_{1}} \wedge \dotsc \wedge \frac{\partial}{\partial z_{p-1}})\big|_{h_{p-1,V}(y_{n})}
\rightarrow \infty.
\]
By   Lemma \ref{lemma: basic lemma1}, we conclude that $X$ is not $(p-1)$-analytically hyperbolic. This contradiction concludes the proof. 
\end{proof}

 \subsection{Pseudo version}

\label{subs: pseudo}

This goal of this section is to obtain the same finiteness result as in the previous part, but by allowing an exceptional locus $\Delta=\Delta_{X,\ell}$ for the $(\dim(X)-1)$ analytic hyperbolicity assumption, and adding accordingly the $\dim(X)$-analytic hyperbolicity condition. Some technical assumptions are to be added on the exceptional locus $\Delta$, and we aim at proving the following slightly more general theorem than Theorem \ref{thm: mainthm2}.
\begin{theorem}
\label{thm: finiteness, pseudo}
Let $X$ be a compact \emph{K\"{a}hler} manifold of dimension $\dim(X)=p$, which is $p$-analytically hyperbolic, and $(p-1)$-analytically hyperbolic modulo $\Delta=\Delta_{X,\ell}$. Suppose furthermore that
\begin{center}
\begin{itemize}
\item{} the $(2p-3)$-Haussdorf measure of $\Delta$ is zero,
\item{} $\Delta$ has a projective neighborhood in $X$ (e.g. $X$, if $X$ is supposed projective).
\end{itemize}
\end{center}
Then $\Aut(X)$ is finite.
\end{theorem}
For the notion of Haussdorf measure, we refer for instance to \cite{eisenman1970intrinsic}.
Let us start with the following elementary proposition:
\begin{proposition}
\label{lemma: maps}
Let $X, Y$ be compact complex manifolds of same dimension, with $X$ $\ell$-analytically hyperbolic modulo $\Delta=\Delta_{X, \ell}$. Suppose there exists a (non-degenerate) holomorphic map $f: Y \rightarrow X$. Then $Y$ is $\ell$-analytically  hyperbolic modulo $f^{-1}(\Delta) \cup \Degen(f)$, where $\Degen(f)$ is the set of points where the differential $df$ is not an isomorphism.
\end{proposition}
\begin{proof}
By hypothesis, there exists $\omega$ a continuous pseudo-metric on (pure multi-vectors of) $\bigwedge^{\ell} TX$, which is a metric on $X\setminus \Delta$ and such that $e_{\ell,X} \geq \omega$ on $X$.
It is straightforward to check that, as $f$ is holomorphic, almost by definition, 
\[
f^{*}e_{\ell,X} \leq e_{\ell, Y}.
\]
Consider then $\omega'\bydef f^{*}\omega$, which is a continuous pseudo-metric on (pure multi-vectors of) $\bigwedge^{\ell} TX$, a metric on $X\setminus \big(\Degen(f) \cup f^{-1}(\Delta)\big)$, and satisfies
\[
\omega' \leq e_{\ell, Y}
\]
on $X$, so that $Y$ is indeed $\ell$-analytically hyperbolic modulo $f^{-1}(\Delta) \cup \Degen(f)$.
\end{proof}

Applying this in particular to $f$ and $f^{-1}$ in $\Aut(X)$ yields the following corollary
\begin{corollary}
Let $X$ be a compact complex variety which is $\ell$-analytically hyperbolic modulo $\Delta=\Delta_{X, \ell}$. Then any automorphism $f\in \Aut(X)$ induces an automorphism $f \in \Aut(X\setminus \Delta)$, i.e. $f(\Delta)=\Delta$.
\end{corollary}
\begin{proof}
By the Lemma \ref{lemma: maps} applied to $f$, $X$ is $\ell$-analytically hyperbolic modulo $f^{-1}(\Delta)$, so that necessarily $\Delta \subset f^{-1}(\Delta)$, as $\Delta=\Delta_{X, \ell}$ is the smallest closed set modulo which $X$ is $p$-analytically hyperbolic. Applying it to $f^{-1}$ yields the reverse inclusion, and proves the corollary.
\end{proof}

Let us recall the notion of meromorphic map (in the sense of Remmert). A meromorphic map $f: X \dashrightarrow Y$ between two complex manifolds $X,Y$ is a correspondance satisfying the following conditions:
\begin{itemize}
\item{} $\forall \ x \in X$, $f(x)$ is a non-empty compact complex subspace of $Y$,
\item{} the graph $G_{f}=\{(x,y) \in X\times Y \ | \ y \in f(x) \}$ is a connected complex subspace of 
$X \times Y$, with $\dim G_{f}=\dim X$,
\item{} there exists a dense subset $X^{*}$ of $X$ such that for all $x \in X^{*}$, $f(x)$ is a single point.
\end{itemize}
Let $\pi:X\times Y \rightarrow X$ be the first projection, $E\subset G_{f}$ the set of points where $\pi$ degenerates, i.e.
\[
E\bydef \{(x,y) \in G_{f} \ | \ \dim f(x) > 0\},
\]
and $S= \pi(E)$, called the \emph{singular locus of $f$}. All we need to know here is that $S$ is a closed complex subspace of $X$ of codimension $\geq 2$, and $f: X \setminus S \rightarrow Y$ is holomorphic (see e.g. \cite{remmert1957holomorphe}). 

Letting $\Delta \subsetneq X$ be a closed subset (of empty interior), a key property that we will need is the following:
\begin{center}
$(\mathcal{P}_{\Delta})$
\ \
Every holomorphic map $f: X \setminus \Delta \rightarrow X$ extends to a meromorphic map $f: X \dashrightarrow X$.
\end{center}

Now, an important result to prove Theorem \ref{thm: finiteness, pseudo} is the following lemma, whose idea can essentially by found in a paper of Yau (\cite{yau1975intrinsic})
\begin{lemma}
\label{lemma: aut}
Let $X$ be a compact complex manifold which is $\dim(X)$-analytically hyperbolic, $(\dim(X)-1)$-analytically hyperbolic modulo $\Delta=\Delta_{X, \ell}$, and satisfies the property $\mathcal{P}_{\Delta}$. Suppose furthermore that $\Delta$ has a projective neighborhood in $X$. Then from any sequence $(f_{n})_{n} \in \Aut(X)^{\NN}$, we can extract a subsequence that converges uniformly on any compact of $X \setminus \Delta$ to an element of $\Aut(X \setminus \Delta)$, which extends to an element of $\Aut(X)$.
\end{lemma}
\begin{proof}
The reasoning of the proof of \ref{thm: mainthm1} works verbatim, and allows to show that $(f_{n})_{n \in \NN}$ admits a converging subsequence to an element $f$ of $\Hol(X\setminus \Delta, X)$.
Similarly, up to extraction, we can suppose that $(f_{n}^{-1})_{n \in \NN}$ converges to $g \in \Hol(X \setminus \Delta, X)$. 

Observe that there exists $m>0$ such that for all $x \in X \setminus \Delta$, the following inequality is satisfied
\[
|\det(df_{x})| \geq m > 0,
\]
where $|.|$ is a fixed continuous metric on $\bigwedge^{p}\Omega_{X}$ ($p=\dim(X))$. Indeed, it is enough to prove that such an inequality is satisfied for $f_{n}$ for all $n \in \NN$ and all $x \in X$.
Arguing by contradiction, by compacity of $X$, we can construct $x_{n_{k}} \rightarrow x$, with $y_{n_{k}}=f_{n_{k}}(x_{n_{k}}) \rightarrow y$, such that $\big|\det\big((df_{n_{k}})_{x_{n_{k}}}\big)\big| \rightarrow 0$. But this implies that
\[
\big|\det\big((df_{n_{k}}^{-1})_{y_{n_{k}}}\big)\big| \rightarrow \infty,
\]
which allows to contradict the $\dim(X)$-analytic hyperbolicity via Lemma \ref{lemma: basic lemma1} in the same fashion that it was done in the course of the proof of Theorem \ref{thm: mainthm1}. Obviously, we can suppose that the same inequality is satisfied for $g$.

By the property $(\mathcal{P}_{\Delta})$, $f$ and $g$ extends meromorphically to $X$. Let then $I_{f}$ (resp. $I_{g}$) be the singular locus of the meromorphic map $f$ (resp. $g$), which is of codimension greater than $2$, and observe that $(df)_{x}$ (resp. $(dg)_{x}$) is invertible for all $x \in X \setminus I_{f}$ (resp. $X \setminus I_{g}$): indeed, this follows from the continuity of $df$ on $X \setminus I_{f}$ combined with the density of $X \setminus \Delta$ in $X \setminus I_{f}$ and the previous inequality we have just shown.
Accordingly, $f$ has a local inverse everywhere on $X \setminus I_{f}$. But observe that $g$ is inverse to $f$ on the dense open set $f(X \setminus \Delta)\cap(X \setminus \Delta)$ so that $g$ must actually extend holomorphically to $f(X \setminus I_{f})$. Therefore, we deduce that $f(X \setminus I_{f}) \subset X\setminus I_{g}$, and by doing the same reasoning for $g$, we obtain aswell that $g(X \setminus I_{g}) \subset X\setminus I_{f}$. But as we necessarily have $g\circ f= \Id_{X \setminus I_{f}}$ and $f\circ g=\Id_{X \setminus I_{g}}$, we actually deduce the following equalities
\begin{equation*}
\label{eq: eq1}
\begin{aligned}
f(X \setminus I_{f})=X \setminus I_{g} \ \text{and} \ g(X \setminus I_{g}) = X\setminus I_{f}.
\end{aligned}
\end{equation*}
Recall that for $S$ any set in $X$, the image of $S$ by the meromorphic map $f$ is defined as
\[
\pi_{2}(G_{f}\cap (S\times X)),
\]
where $\pi_{2}$ is the projection onto the second factor. Since $f(X)=X$ ($G_{f} \subset X \times X$ is closed and so is its image by $\pi_{2}$ as it is a proper map; the equality follows as this projection contains a dense open set of $X$), by writing
\[
X=f(X)=f(X \setminus I_{f}) \cup f(I_{f})=(X \setminus I_{g}) \cup f(I_{f}),
\]
we deduce that $I_{g} \subset f(I_{f})$, and similarly, we obtain that $I_{f} \subset g(I_{g})$.

The goal is to prove that $I_{f}=I_{g}=\emptyset$. Arguing by contradiction, we suppose that it is not empty.
Let then $N$ be a projective neighborhood of $\Delta$ (which exists by hypothesis), and $H$ an irreducible hypersurface of $N$ such that $N \setminus H$ is Stein, and $H\cap I_{f} \neq I_{f}$. Observe that $f(N)$ is an open set, and that $f(H)$ is a subvariety of $f(N)$. Furthermore, defining
\[
A\bydef I_{g} \cap \big(f(N)\setminus f(H)\big),
\]
we prove that $A$ is not the empty set.
Indeed, otherwise we have the inclusion $I_{g} \subset f(H)$, so that, using the above, we deduce that
\[
I_{f} \subset g(I_{g}) \subset g(f(H)).
\]
But $g(f(H))$ is a closed irreducible analytic set, which contains the irreducible open set $H \setminus I_{f}$, so that necessarily $g(f(H))=H$, and thus $I_{f} \subset H$, which was excluded by construction.

To conclude, observe that $g$ maps the open set $\big(f(N)\setminus f(H)\big) \setminus A$ into the Stein manifold $N\setminus H$, with $\codim(A)\geq 2$, so that $g$ extends to a holomorphic map through $A\neq \emptyset$ (see e.g.\cite{Andreotti_extension}[Theorem 2, p.316]), which yields a contradiction to the very definition of $I_{g}$. Therefore, we deduce that $I_{g}=\emptyset$. Similarly, we obtain that $I_{f}=\emptyset$. Accordingly, $f$ and $g$ are defined on $X$, and as 
\[
f\circ g=\Id_{X \setminus I_{g}}=\Id_{X}=\Id_{X \setminus I_{f}}=g\circ f,
\]
we deduce that $f$ is indeed in $\Aut(X)$.
And by Proposition \ref{lemma: maps}, we know that $f$ induces aswell an element in $\Aut(X \setminus \Delta)$.

\end{proof}

We are now in position to prove the announced theorem.

\begin{proof}[Proof of Theorem \ref{thm: finiteness, pseudo}]
We equip $X$ with an hermitian metric, with which we define the topology of uniform convergence on $C^{0}(X)$, and the topology of uniform convergence on any compact of $X \setminus \Delta$ on $C^{0}(X \setminus \Delta)$. Note that these two topologies are metrizable. 
As the closed set $\Delta$ is of $(2p-3)$-Haussdorf measure $0$, it is in particular of empty interior. Furthermore, as $X$ is compact and \emph{K\"{a}hler}, a theorem of Siu (\cite{siu1975extension}, Theorem 1) ensures that the property $\mathcal{P}_{\Delta}$ is indeed satisfied for $X$.

Now, by Lemma \ref{lemma: aut}, there is a continuous injective restriction map
\[
r: \Aut(X) \rightarrow \Aut(X \setminus \Delta),
\]
whose image $G=r(\Aut(X))$ is compact. But it is also easily seen that $G$ is a subgroup of $\Aut(X\setminus \Delta)$, therefore acting on $X\setminus \Delta$. By a theorem of Bochner-Montgomery \cite{bochner1946locally}, $G$ has a structure of a compact real Lie group. But $G$ is at most countable (because $\Aut(X)$ is, see Lemma \ref{lem:discreteness_measure}), so that it must be of dimension zero, and as it is compact, it must be finite. Accordingly, we deduce that $\Aut(X)$ is finite, which finishes the proof.
\end{proof}

As closed analytic subset of codimension $\geq 2$ in a compact complex variety $X$ of dimension $p$ has zero $(2p-3)$-Haussdorf measure, and as projective varieties are K\"{a}hler, we immediately deduce the following corollary
\begin{corollary}
Let $X$ be a complex projective manifold of dimension $\dim(X)=p$, which is $p$-analytically hyperbolic, and $(p-1)$-analytically hyperbolic modulo $\Delta$, where $\Delta=\Delta_{X, \ell}$ is an analytic subset of $X$ of codimension $\geq 2$. Then $\Aut(X)$ is finite.
\end{corollary}

\subsection{Conditions for intermediate (analytic) hyperbolicity}

\subsubsection{Algebraic condition}
Starting from the classic fact that a projective manifold $X$, with $\Omega_{X}$ ample, is hyperbolic (and hence has $\Aut(X)$ finite),  Demailly proved in \cite{Demailly} the following generalization:
\begin{theorem}[Demailly, \cite{Demailly}]
Let $X$ be a projective manifold of dimension $p$. Let $1\leq \ell \leq p$, and suppose that
\[
\bigwedge^{\ell} \Omega_{X}
\]
is ample. Then $X$ is $\ell$-analytically hyperbolic.
\end{theorem}
In \cite{Etesse2019Ampleness}, the author studied complete intersections in projective spaces satisfying such a positivity property, and proved that a general complete intersection $X \subset \P^{N}$ of $c \geq \frac{N}{\ell +1}$ hypersurfaces of large enough degrees (with an explicit very large bound) has its $\ell$th exterior power of its cotangent bundle $\bigwedge^{\ell} \Omega_{X}$ ample.
\subsubsection{Curvature condition}
It is well known that negative curvature properties imply hyperbolic properties, which therefore should imply finiteness results. For example, a compact complex manifold endowed with a \emph{smooth hermitian metric} of negative holomorphic sectional curvature is known to be hyperbolic (see \cite{Kobayashi}, and therefore has finite automorphisms group. A natural question is to extend this result to the setting of a \emph{singular metric}. Regarding this question, a recent work of Guenancia \cite{Guenancia} generalizing results of Wu-Yau \cite{WY} provides the following corollary.

\begin{theorem}[Guenancia \cite{Guenancia}]\label{sing-finite}
Let $X$ be complex projective manifold endowed with a smooth closed semipositive $(1,1)$-form $\omega$ such that there exists a Zariski open subset where $\omega$ defines a K\"ahler metric which has uniformly negative holomorphic sectional curvature . Then $X$ is of general type (and therefore has finite automorphisms group).
\end{theorem}

It is natural to try to extend these kind of results to intermediate hyperbolicity. Graham and Wu have introduced in \cite{graham1985} curvature conditions which imply $\ell$-analytic hyperbolicity. Namely, Theorem 4.5 in \cite{graham1985} show that \emph{strongly negative $\ell$-th Ricci curvature} (see loc. it. for a definition) imply $\ell$-analytic hyperbolicity. 

An immediate corollary of this statement and Theorem \ref{thm: mainthm1} is the following.

\begin{corollary}
Let $X$ be complex compact hermitian manifold with strongly negative $\ell$-th Ricci curvature for some $1\leq \ell \leq \dim X$. Then $\Aut(X)$ is finite.
\end{corollary} 

\begin{proof}
If $\ell=\dim X$ then $X$ is of general type and we are done.

If $\ell < \dim X$ then Theorem 4.5 of \cite{graham1985} implies that $X$ is $\ell$-analytically hyperbolic, which implies that $X$ is $(\dim X - 1)$-analytically hyperbolic by Proposition \ref{prop:l_to_lplusone}.
Then Theorem \ref{thm: mainthm1} implies that $\Aut(X)$ is finite.
\end{proof}

\begin{remark}
As noted by Graham and Wu (\cite{graham1985} p. 638), it is unknown if the existence of a Hermitian metric of strongly negative $\ell$-th Ricci curvature implies the existence of one with strongly negative $(\ell+1)$-th Ricci curvature.
\end{remark}

\begin{remark}
It would be interesting to generalize this statement to the setting of singular metric such as in Theorem \ref{sing-finite} above. It does not seem obvious how to use Theorem \ref{thm: mainthm2} in this context. More precisely, the following question seems interesting to us. Suppose that $X$ is a complex compact hermitian manifold with pseudo-strongly negative $\ell$-th Ricci curvature for some $1\leq \ell \leq \dim X$. Then is it true that $\Aut(X)$ is finite?
\end{remark}

 \section{Intermediate Picard hyperbolicity}
 \label{section: Picard}
Motivated by the extension property $(\mathcal{P}_{\Delta})$ introduced during the proof of Theorem \ref{thm: mainthm2}, we generalize the notion of Picard hyperbolicity introduced in \cite{JKuch} and \cite{deng2020big} to the context of intermediate hyperbolicity.

\begin{definition}
Let $X$ be a compact complex manifold and $\Delta$ a proper closed subset of empty interior. $X$ is said \emph{$\ell$-Picard hyperbolic} modulo $\Delta$ if for every $p,q \in \NN$, $p+q=\ell$, every non-degenerate holomorphic map $f: \BB^p \times (\BB^*)^q \to X$ whose image is not contained in $\Delta$ extends meromorphically to a map $f:\BB^{\ell} \dashrightarrow X$.
\end{definition}

\begin{remark}
\label{remark: bimero}
Note that if $X$ is $\ell$-Picard hyperbolic, then so is any compact complex manifold $X'$ bimeromorphic to $X$, as we are interested in \emph{meromorphic} extensions.
\end{remark}

If $X$ is moreover supposed to be \emph{K\"ahler}, then in order to obtain the pseudo $\ell$-Picard hyperbolicity, it is enough to check the conditions only for $q=1$, i.e. for non-degenerate holomorphic maps $f: \BB^{\ell-1} \times \BB^* \to X$ whose image is not contained in the exceptional locus:
\begin{lemma}
\label{lemma: Picard}
Let $X$ is a compact K\"ahler manifold, and let $1\leq \ell \leq \dim(X)$. Then $X$ is $\ell$-Picard hyperbolic modulo $\Delta$ if and only if  any non-degenerate holomorphic map $f: \BB^{\ell-1} \times \BB^{*} \to X$  whose image is not included in $\Delta$ extends meromorphically to $\BB^{\ell}$.
\end{lemma}
\begin{proof}
This is an application of the extension theorem of Siu \cite{siu1975extension}[Theorem1], which implies in particular that any meromorphic map (without any further condition) from $\BB^{\ell} \setminus E$ to $X$ extends meromorphically as soon as $E$ is a closed subset of codimension $\geq 2$. Indeed, let $q\geq1$, and let $f: \BB^{p}\times (\BB^{*})^{q} \rightarrow X$, $p+q=\ell$, be a non-degenerate holomorphic map whose image is not included in $\Delta$. Extending $f$ meromorphically is a local question, and observe that around a point $z_{0} \in \BB^{p} \times (\BB^{*})^{q-j}\times \{0\}^{j}$, $0 \leq j \leq q$, up to a renormalization, $f$ can be interpreted as a holomorphic map from $\BB^{\ell-j} \times (\BB^{*})^{j}$ to $X$. If $j=1$, we can extend $f$ meromorphically by hypothesis. If $j>1$, an obvious induction shows that we can extend meromorphically $f$ to $\BB^{\ell-j} \times \BB^{j} \setminus {(0, \dotsc, 0)}$. Now, as $j>1$, we can apply Siu's extension theorem, so that $f$ actually extends meromorphically to $\BB^{\ell}$, and we are done.
\end{proof}

   One also sees from \cite{siu1975extension} that if $X$ is $1$-Picard hyperbolic then $X$ is $\ell$-Picard hyperbolic for all $1 \leq \ell \leq \dim X$ (see \cite{deng2020big}[Prop. 3.4]).  In fact, one has the more general property that if $X$ is pseudo-$\ell$-Picard hyperbolic then $X$ is pseudo-$(\ell+1)$-Picard hyperbolic as shown in the next proposition.

\begin{proposition}
\label{Picard: ell to ell plus one}
Let $X$ be a compact K\"ahler manifold. If $X$ is $\ell$-Picard hyperbolic modulo $\Delta$ then $X$ is $\ell+1$-Picard hyperbolic modulo $\Delta$.
\end{proposition}

\begin{proof}
By Lemma \ref{lemma: Picard}, in order to prove the $(\ell+1)$-Picard hyperbolicity modulo $\Delta$, it is enough to show that any non-degenerate holomorphic map
\[
f: \BB^{\ell} \times \BB^{*} \rightarrow X
\]
whose image is not included in $\Delta$ extends to a meromorphic map. Since $f$ is non-degenerate, there exists a dense Zariski open set $U$ on which the differential $df$ is of maximal rank. Denoting $\pi: \BB \times (\BB^{\ell-1} \times \BB^{*}) \rightarrow \BB$  the projection on the first factor, it is easily seen that for any $z$ in the open set $\pi(U)$, the restriction map  $f(z,.): \BB^{\ell-1}\times \BB^{*} \rightarrow X$ remains non-degenerate. Furthermore, one can find $O \subset \pi(U)$ open set such that $f(z,.)$ is not included in $\Delta$ for every $z \in O$. In particular, since $X$ is $\ell$-Picard hyperbolic modulo $\Delta$, this map extends to a meromorphic map, which \`a fortiori implies that for any $w$ in the open set $p(\pi^{-1}(O))$, where $p$ is the projection $p: \BB^{\ell}\times \BB^{*} \rightarrow \BB^{\ell}$, the map $f(w,.): \BB^{*} \rightarrow X$ extends to a meromorphic map.

Now, by invoking the extension result of Siu \cite{siu1975extension}[result (*), p.442], this actually implies that the holomorphic map $f: \BB^{\ell}\times \BB^{*} \rightarrow X$ extends to a meromorphic map, which indeed shows that $X$ is $(\ell+1)$-Picard hyperbolic modulo $\Delta$.
\end{proof}

Examples of Picard hyperbolic varieties are provided by projective manifolds whose cotangent bundles satisfy some positivity properties as shown by the following result of Noguchi \cite{Noguchi77} which generalizes the classical fact that varieties of general type are $\dim X$-Picard hyperbolic.

\begin{theorem}[Noguchi]
Let $X$ be a smooth projective variety over $\CC$ and let $\Delta\subset X$ be a proper Zariski-closed subset.
If  $\bigwedge^p \Omega^1_X$ is  ample modulo $\Delta$, then $X$ is $p$-Picard hyperbolic modulo $\Delta$.
\end{theorem}

It is natural to try to give positivity conditions which will still ensure some intermediate Picard hyperbolicity. In this direction, one can state the following corollary of a recent result of \cite{CP19}.
\begin{theorem}[Campana, P\u{a}un]
Let $X$ be a smooth projective variety over $\CC$ such that some tensor power of $\Omega^1_X$ is big. Then $X$ is of general type, in particular $\dim X$-Picard hyperbolic.
\end{theorem}

\begin{remark}
A natural question is to provide negative curvature conditions (such as negative holomorphic sectional curvature for $\ell=1$) which will guarantee Picard hyperbolicity. Graham and Wu have introduced in \cite{graham1985} curvature conditions which imply $\ell$-analytic hyperbolicity. It seems unknown whether such conditions guarantee $\ell$-Picard hyperbolicity. In the singular setting, it is shown recently in \cite{DLSZ} that a projective manifold equipped with a singular metric of negative holomorphic sectional curvature (in the sense of currents) is $1$-Picard hyperbolic modulo the singular set of the metric.
\end{remark}

We now justify why the notion of intermediate Picard hyperbolicity is suitable for dealing with extension problems such as Conjecture \ref{conjecture: extension}. (See also \cite{deng2020big}[Proposition 3.4]).
\begin{proposition}
\label{prop: Picard and extension}
Let $X$ be a projective manifold which is $\ell$-Picard hyperbolic, $1\leq \ell \leq \dim X$. Let $Y$ be a projective manifold of dimension $\dim(Y) \geq \ell$, and let $H \subset Y$ be a proper closed subset. Then any non-degenerate holomorphic map
\[
f: Y \setminus H \rightarrow X
\]
extends meromorphically through $H$.
\end{proposition}
\begin{proof}
As we can always take an hypersurface containing the proper closed subset $H$, it is enough to prove the result when $H$ is an hypersurface. Having in mind Remark \ref{remark: bimero}, one sees that, by using a resolution of singularities (see e.g. \cite{Lazzie1}[Theorem 4.3.1]), it is enough to prove the result when $H$ is a simple normal crossing divisor. But in this situation, the local picture around a point $h \in H$ is clear: there exists a neighborhood $U_{h} \subset Y$ of $h$
such that
\[
U_{h} \setminus (H\cap U_{h}) \simeq \BB^{p} \times (\BB^{*})^{q}
\]
for some $p,q$ such that $p+q=\dim(Y) \geq \ell$. Since $X$ is $\ell$-Picard hyperbolic, we know by Proposition \ref{Picard: ell to ell plus one} that it is also $\dim(Y) \geq \ell$-Picard hyperbolic, so that any non-degenerate holomorphic map
\[
U_{h} \setminus (H\cap U_{h}) \rightarrow X
\]
extends meromorphically through $H \cap U_{h}$. As this is true for any $h \in H$, this shows in particular the announced result.
\end{proof}

 We now end this section by proving the announced Theorem \ref{thm: mainthm3}.
\begin{proof}[Proof of Theorem \ref{thm: mainthm3}]
In view of the scheme of proof of Theorem \ref{thm: mainthm2}, it is enough to check that under the assumption of $\dim(X)$-Picard hyperbolicity, Conjecture \ref{conjecture: extension} is satisfied. But this is in particular the content of Proposition \ref{prop: Picard and extension}, so that the result follows.
\end{proof}
 
 From all of this, it is obviously tempting to make the following conjecture:
 \begin{conjecture}
 A projective manifold $X$ is $\ell$-Picard hyperbolic if and only if it is $\ell$-analytically hyperbolic.
 \end{conjecture}
However, none of the direction seems easy at first. A first step that seems feasible would be to prove that $(\ell-1)$-analytic hyperbolicity implies $\ell$-Picard hyperbolicity.

\bibliography{refsperiod}{}
\bibliographystyle{alpha}

\end{document}